\documentclass{article}[12pt]

\usepackage{amsfonts}
\usepackage{amsmath}
\usepackage{amssymb}
\usepackage{amscd}
\usepackage{amsthm}
\usepackage{indentfirst}
\usepackage[hmargin=3cm,vmargin=3cm]{geometry}

\newtheorem{thm}{Theorem}
\newtheorem{lma}{Lemma}
\newtheorem{cpt}{Computation}
\newtheorem{cor}{Corollary}

\theoremstyle{definition}

\newtheorem{df}{Definition}[subsection]
\newtheorem{ft}{Fact}

\theoremstyle{remark}
\newtheorem*{pf}{Proof}
\newtheorem{rmk}{Remark}[subsection]
\newtheorem*{exa}{Example}
\newtheorem{qtn}{Question}

\newcommand{\R}{{\mathbb{R}}}
\newcommand{\Z}{{\mathbb{Z}}}
\newcommand{\C}{{\mathbb{C}}}
\newcommand{\bs}{\bigskip}
\newcommand{\ra}{{\rightarrow}}

\def\mVol{\text{Vol}(M,\omega^n)}

\begin{document}

\title{\textbf{Remarks on invariants of Hamiltonian loops}\\}

\author{\textsc{Egor Shelukhin$^{1,2}$}
\\
School of Mathematical Sciences \\
Tel Aviv University\\
69978 Tel Aviv, Israel \\
\tt egorshel@post.tau.ac.il}

\footnotetext[1]{This paper is a part of the author's Ph.D. thesis,
being carried out under the supervision of Prof. Leonid Polterovich,
at Tel-Aviv University.} \footnotetext[2]{Partially supported by the
Israel Science Foundation grant $\#$ 509/07}

\date{May 2009}

\maketitle
\begin{abstract}
In this note the interrelations between several natural morphisms on
the $\pi_1$ of groups of Hamiltonian diffeomorphisms are
investigated. As an application, the equality of the (non-linear)
Maslov index of loops of quantomorphisms of prequantizations of $\C
P^n$ and the Calabi-Weinstein invariant is shown, settling
affirmatively a conjecture by A.Givental. We also prove the
proportionality of the mixed action-Maslov morphism and the Futaki
invariant on loops of Hamiltonian biholomorphisms of Fano Kahler
manifolds, as suggested by C.Woodward. Finally, a family of
generalized action-Maslov invariants is computed for toric manifolds
via barycenters of their moment polytopes, with an application to
mass-linear functions recently introduced by D.McDuff and S.Tolman.

\end{abstract}

%

\section{Introduction and main results:}

The topology of groups of Hamiltonian diffeomorphisms of symplectic
manifolds and of the closely related groups of quantomorphisms of
their prequantizations is an intriguing area of modern symplectic
geometry. In the present note we study the $\pi_1$ of these groups
by investigating the interrelations between several natural
morphisms $\pi_1 \to \R$. Our main results are the following:

\begin{enumerate}
\item
We show that for the complex projective space $\C P^n$ the
(non-linear) Maslov index on loops of
quantomorphisms~\cite{givental} of its prequatizations is
proportional to the Calabi-Weinstein~\cite{calabi,weinstein}
homomorphism (Theorem~\ref{equality of m and cw} below). This
settles in positive a conjecture by A.Givental~\cite{givental}.

\item
Following a suggestion by C.Woodward~\cite{woodward}, we prove that
the mixed action-Maslov invariant~\cite{polterovich loops} on Fano
Kahler manifolds is proportional to the Futaki invariant~\cite{f}.

\item
Finally, we compute a family of generalized action-Maslov
invariants~\cite{lmp,kj,symplectures} on toric manifolds via
barycenters of their moment polytopes, and present applications to
the notion of mass-linear functions recently introduced by D.McDuff
and S.Tolman in~\cite{poly with ml}.
\end{enumerate}

\subsection{Mixed action-Maslov invariants.}

\begin{df}\label{MAM}

{ \it (Mixed Action-Maslov homomorphism)}

Let $(M,\omega)$ be a compact spherically monotone symplectic
manifold. That means that $[\omega]=\kappa c_1(M)$ on $\pi_2(M)$
for some positive $\kappa \in \R$. And denote by
$Ham=Ham(M,\omega)$ the group of Hamiltonian diffeomorphisms of
$(M,\omega)$.

\bs

 Then the Mixed Action-Maslov homomorphism $I:\pi_1(Ham) \to \R
$ introduced by L.Polterovich in~\cite{polterovich loops} is defined
as follows:

\bs

Given a class $a \in \pi_1 (Ham, Id)$ \\a) Choose a representative
$\gamma:=\{\phi_t\}_{t \in \R / \Z}$ of $a$. \\b) Choose a smooth
time dependent Hamiltonian for $\gamma$ : $\{F_t\}_{t \in \R / \Z}$
, that is normalized by the condition that for all $t$, \ $\int_M
F_t \, \omega ^n = 0 \ $. \\c) Choose a point $p$ in $M$. Then the
path $\alpha_t :=\{\phi_t (p) \}_{t \in \R / \Z}$ is contractible by
Floer Theory.
\\d) Choose a disk $D$ filling the path $\alpha$.

\bigskip
Then define $I(\gamma):= \int_D \omega - \int_0^1 F_t(\alpha_t)dt
- \kappa \cdot Maslov/2 (\gamma_{*_p}^D) $.

\bigskip
 Here, $\gamma_{*_p}^D$ is the path of symplectic linear
operators obtained when considering the linear maps
$\phi_{t_{*_p}}:T_pM \ra T_{\phi_t(p)}M$ and trivializing the
bundle $(TM, \omega)$ symplectically along the disk $D$.

\bigskip

{\bf Convention:} {\it hereafter, our normalization of the Maslov
index is such that the Hopf flow $\{R_{2 \pi t}:z \mapsto e^{2 \pi
t}z\}_{0 \leq t \leq 1}$  on $(\C,\omega_{std})$ has Maslov index
2.}

\bigskip
This value does not depend on the choices a)-d), and defines a
homomorphism $I:\pi_1(Ham) \to \R $. The most essential part is the
independence on choice d) and it follows from the assumption that
our manifold is spherically monotone. (see~\cite{polterovich loops})
\end{df}

\begin{rmk}
This invariant provides a lower bound on the asymptotic Hofer norm
(\cite{polterovich loops}), and its vanishing is necessary and
sufficient for asymptotic spectral invariants to descend to the
group $Ham$ from its universal cover (\cite{monodromy in ham}). It
is also related to hamiltonianly non-displaceable fibers of moment
maps of Hamiltonian torus actions (\cite{rigid}).
\end{rmk}

The homomorphism $I$ is known to vanish identically in several
cases:

\begin{ft}\label{surfaces}
{\it} Since I is a homomorphism to $\R$, it vanishes whenever
$\pi_1(Ham)$ contains only elements of finite order. This holds
for 1. all compact surfaces and 2. the product of two spheres with
equal areas.
\end{ft}

\begin{ft}\label{Seidel}
For $\C P^n$, $I$ is also known to vanish, by a non-trivial argument
involving the Seidel representation (view~\cite{rigid} Theorem
1.11,~\cite{calabi qm} section 4.3). The link between the two is
established by the reinterpretation of $I$ as the homogeneization of
the valuation of the Seidel representation.
\end{ft}

To fact~\ref{Seidel} the following corollary holds:

\bs

Consider $(M,\omega)=(\C P^n, \omega_{FS})$, where the Fubini-Study
form is normalized to represent the generator of the integer
cohomology. Let $(P,\alpha):= (S^{2n+1}, \frac{pdq-qdp}{2\pi})$ be a
prequantization space of $M$. Denote $H := Ham(M,\omega)$ and $Q :=
Quant(P,\alpha)$. In this case two homomorphisms are defined: \bs

The first:

\begin{df}\label{Calabi-Weinstein}
{\it (Calabi-Weinstein homomorphism~\cite{calabi,weinstein})}

Given a path $\{\widehat{\phi_t}\}_{0 \leq t \leq
1};\;\widehat{\phi_0}=Id$ representing a class in $\widetilde{Q}$,
let $h_t$ be its contact Hamiltonian, considered as a function on
the base $\C P^n$.

Then
 $$cw(\{\widehat{\phi_t}\}):= \int_0^1 dt \int_M h_t \omega^n $$

does not depend on homotopy with fixed endpoints and determines a
homomorphism $\widetilde{Q} \to \R$.

Its restriction to $\pi_1(Q)$ is considered.
\end{df}

\begin{rmk}
This definition works for the general situation of prequantization
spaces.
\end{rmk}

\bs

 And the second:

\begin{df}\label{Nonlinear-Maslov}
{\it (Nonlinear Maslov index for loops~\cite{givental})}

Consider a loop $\{\widehat{\phi_t}\}_{t \in S^1}$ that represents
a class $b \in \pi_1(Q)$.

Denote by $\{\Phi_t\}_{t \in S^1}$ its lifting to a loop of
homogenous Hamiltonian diffeomorphisms of $\C^{n+1} \setminus
\{0\}$. (note that $\C^{n+1} \setminus \{0\}$ is the symplectization
of $(P,\alpha)$)

For a point $y \in \C P^n$ denote by $P_y \cong S^1$ the fiber of
$P$ over $y$, and by $SP_y \cong \C ^* = \C \setminus \{0\}$ the
fiber of $SP$ over $y$.

Let $\displaystyle m(\{ {\Phi_t}_{*_1}\}_{t \in S^1})$ be the
Maslov index of the linearization of the lift $\{\Phi_t\}$ at a
point $1 \in P_y \subset SP_y \cong \C ^*$ in the fiber over y,
when $T (\C^{n+1} \setminus \{0\})$ is naturally identified with
$\C^{n+1} \setminus \{0\} \times \C^{n+1}$ by the linear
structure.

Then,
$$\mu(\{ \widehat{\phi_t} \}):=  \displaystyle m(\{ {\Phi_t}_{*_1}\})$$

depends only on the class $b \in \pi_1(Q)$ and is a homomorphism
$\pi_1(Q) \to \R$.
\end{df}

\begin{rmk}
Although this invariant does not a-priori extend to a homomorphism
$\widetilde{Q} \to \R$, it is known to extend to a \emph{homogenous
quasimorphism} $\widetilde{Q} \to \R$ (\cite{givental,gabi}).
\end{rmk}

\bs

 The two invariants of loops we've defined happen to be equal:

\begin{thm} \label{equality of m and cw}

On $\pi_1(Q)$,

$$ \frac{1}{\mVol} cw=\frac{1}{2(n+1)} \mu $$

\end{thm}

\begin{exa}

By taking a loop with Hamiltonian $h_t \equiv 1$, one obtains the
standard Reeb rotations of $\C ^{(n+1)} \setminus \{0\}$. The Maslov
index the linearization at any point is $2(n+1)$. Hence, in this
case $\frac{1}{2(n+1)} \mu = 1$, and $\frac{1}{\mVol} cw = 1$. So
the equality holds in this case.

\end{exa}

\begin{rmk}
The given equality enables the extension of $\mu$ to a homomorphism
$\widetilde{Q} \to \R$.
\end{rmk}

\begin{rmk}
This theorem answers a conjecture posed by A.Givental
in~\cite{givental}. The above formula transforms to the one
in~\cite{givental} as follows:

Note that while the Nonlinear Maslov index on loops does not vary
when the symplectic form is positively scaled, the
Calabi-Weinstein homomorphism does. In fact the right hand side of
the equality proved varies homogeneously of degree $1$ in the
scale parameter. So that after our formula is established, we can
scale the form so that $Vol(M,\omega^n) = 1$, to obtain the
required formula.

The factor $\frac{1}{2\pi}$ which we omitted comes purely from the
definition of the Calabi-Weinstein homomorphism (and relates to
the eternal question of whether the length of $S^1$ is $1$ or
$2\pi$).
\end{rmk}

\begin{rmk}\label{formula on MICW}
The main formula in the proof is:
$$ \frac{1}{2(n+1)} m (\widehat{\gamma})  - \frac{1}{\mVol} cw(\widehat{\gamma}) = - I(\gamma) $$
for all loops $\widehat{\gamma}$ in $Q$, and their corresponding
loops $\gamma$ in $H$.

This formula works for prequantizations of arbitrary (spherically)
monotone integral symplectic manifolds, with appropriate
modifications, since for their symplectizations $c_1=0$ on
spheres.

\end{rmk}

\begin{rmk}
Another extension of $\mu$ to a quasimorphism
$\mathfrak{T}:\widetilde{Q} \to \R$, is obtained by changing $m$ to
the linear Maslov quasimorphism
 on the universal cover of the symplectic group (refer to \cite{rigid}),
and averaging over the sphere with the natural
measure $\alpha\wedge(d\alpha)^n$  w.r.t. the point of linearization.\\
This construction works in the general case of prequantization
spaces, with simple modifications, and the quasimorphism thus
obtained is equal to the one ($\mathfrak{S}:\widetilde{H} \to \R$)
from~\cite{py qm} as follows:

$$ \frac{1}{2(n+1)} \mathfrak{T}  - \frac{1}{\mVol} cw = -  \frac{1}{\mVol} \mathfrak{S}$$

And this links between the quasimorphism from~\cite{py qm} and the
one from from~\cite{entov quasimorphism}.
\end{rmk}

\begin{rmk}
An equality analogous to Theorem~\ref{equality of m and cw} works
for other prequantizations $P$ of $(\C P^n, \omega_{FS})$. The
modification for $c_1(P)=p[\omega_{FS}]$ is as follows: one
considers the subgroup $Q_p \subset Quant(S^{2n+1})$ of
quantomorphisms that commute with the action of the group of roots
of unity of order $p$. This group is a $p$-to-$1$ cover of
$Quant(P)$. Therefore, one can lift every $p$-th power $\gamma^p$ of
a loop $\gamma$ in $Quant(P)$ to a loop $\widetilde{\gamma^p}$ in
$Q_p \subset Quant(S^{2n+1})$. One then defines
$\mu(\gamma):=\frac{1}{p} \mu(\widetilde{\gamma^p})$.

\end{rmk}

 Theorem~\ref{equality of m and cw} has the following consequence:

\begin{cor} \label{m cw and pi_1}

For $(M,\omega)=(\C P^n, \omega_{FS})$, $$\pi_1(H) =
{\Z}/{{(n+1)}\Z} \oplus Ker(cw|_{\pi_1(Q)})$$

where ${\Z}/{{(n+1)}\Z}=\pi_1(PU(n+1))$

\end{cor}

This reproves the well-known fact that for $\C P ^n$,
$\Z/{{(n+1)}\Z}$ embeds into $\pi_1(Ham)$.

Several proofs are already known - e.g.~\cite{reznikov,seidel,poly
with ml,pi_1 injects}

\begin{rmk} Also, from this follows a curious fact that the $(n+1)$-st power of any loop in $H$
based at identity lifts to a loop in $Q$ (based at identity).
\end{rmk}

\begin{rmk}
The splitting works because the groups are abelian. A proof of
this corollary can be found in section~\ref{proof of cor to m cw}.
\end{rmk}

The mixed action-Maslov invariant has a share of generalizations,
which are defined for all, not necessarily spherically monotone,
symplectic manifolds.

\begin{df}\label{generalized I}
{\it (Generalized mixed action-Maslov invariants,
along~\cite{lmp,kj,symplectures})}

 Let $M:=(M, \omega)$ be a symplectic manifold of
dimension 2n. $Ham:= Ham(M)$ - the group of Hamiltonian
diffeomorphisms.

The generalized mixed Action-Maslov Invariants are a family of
homomorphisms $\pi_1(Ham,Id) \rightarrow \R$, that are defined
similarly to the usual Chern numbers of complex vector bundles:

To a Hamiltonian loop based at identity, $\phi = \{\phi_t\}_{t \in
S^1} \subset Ham$ one can associate by the clutching construction a
Hamiltonian fiber bundle $\pi:P \rightarrow S^2$ over $S^2 = \C P^1$
with fiber $M$. This gives a one-to-one correspondence between
isomorphism classes of such Hamiltonian bundles over the 2-sphere
and $\pi_1(Ham, Id)$. The group operation in $\pi_1$ corresponds to
the "fibered"-connected sum operation of bundles up to isomorphism.
(\cite{lmp}).

For $a \in \pi_1(Ham, Id)$ we'll denote by $P_a$ the corresponding
isomorphism class of bundles (or rather it's representative).

Let $V:=Ker(\pi_*:TP \rightarrow TS^2) \subset TP$ be the vertical
(sub)bundle. It is a symplectic vector bundle of rank 2n, and so,
by the canonical-up-to-homotopy choice of a compatible almost
complex structure, it has well defined ("vertical") Chern classes
$$\{c_l:=c_l(V) \in H^{2l}(P,\Z)\}_{0 \leq l \leq n}. $$

Let $u \in H^2(P, \R)$ denote the \emph{coupling class} of the
Hamiltonian vector bundle $P$, defined by the conditions $$1)
u|_{fiber} = \omega; \,\,\, 2)\int_{fiber} u^{n+1} = 0
\Leftrightarrow u^{n+1} = 0.
$$

\begin{rmk}
The equivalence in 2) is specific to the case where the base is
2-dimensional (consult e.g.~\cite{kj})
\end{rmk}

\begin{rmk}
More generally, these characteristic classes can also be defined
as elements in the cohomology of corresponding classifying spaces.
(\cite{kj})
\end{rmk}

\bs

 Now choosing, as in the definition of Chern numbers, a
monomial $\alpha=(c_1)^{i_1} \cdot ... \cdot (c_n)^{i_n} u^j$ of
degree $2n+2$, define $I_{\alpha}: \pi_1(Ham,Id) \rightarrow \R$ by

$$I_{\alpha}(a):= \displaystyle
\int_{P_a} (c_1)^{i_1} \cdot ... \cdot (c_n)^{i_n} u^j.$$

\bigskip

$I_\alpha$ is a homomorphism, as is seen from the compatibility of
the clutching construction with the group operations (turn
to~\cite{lmp}).

\end{df}

\bs

A simple computation shows that in the monotone case, $I_{(c_1)^L
u^{n+1-L}}$ are all proportional to $I$.

\begin{cpt}\label{proportional}
For a monotone symplectic manifold $(M,\omega)$, with
$[\omega]=\kappa c_1(M)$,

$$I_{(c_1)^L
u^{n+1-L}} = - \frac{L}{\kappa^L} \cdot \mVol \cdot I$$
\end{cpt}

In particular, $I_{(c_1)u^n}= - \frac{1}{\kappa} \,\mVol \cdot I$,
and in the case $\kappa=1$,  $I_{(c_1)^{n+1}}= - (n+1) \,\mVol
\cdot I$.

\begin{proof}
For monotone symplectic manifolds, $u =\kappa c_1 + I \cdot
\pi^*(a)$, where $a$ is the generator of $H^2(S^2,\R)$, talking in
the language of symplectic fibrations~\cite{polterovich loops}. So
that $c_1 = \frac{1}{\kappa}(u - I \cdot \pi^*(a))$. Plugging this
in, we readily obtain the formula.
\end{proof}

\subsection{Futaki invariants} \label{futaki invts}

The mixed action-Maslov invariant and its generalizations happen
to be related to Futaki invariants.

\begin{df}\label{Futaki basic}
{\it(Futaki Invariant)}

Let $(M,\omega,J)$ be a compact Fano manifold. That is to say Kahler
and $[\omega]=c_1(M)$ . Denote by $\mathfrak{h}(M)$ the complex Lie
algebra of holomorphic sections of $T^{(1,0)}M$. Then the
\emph{Futaki Invariant} $F: \mathfrak{h}(M) \rightarrow \C$
introduced by A.Futaki in~\cite{f} is defined as follows:

By the Kahler property, $\omega$ is a closed, real, (1,1)-form in
the class $c_1(M)$. And the Ricci form $Ric(\omega)$ of $\omega$
is also a closed, real, (1,1)-form in the class $c_1(M)$ (by the
symmetry of the curvature tensor, and by Chern- Weil theory of
characteristic classes). Therefore by the $dd^c$-Lemma, there
exists a function $f_{\omega} \in C^\infty(M,i \R)$ unique up to
constant, such that,

$$Ric(\omega)- \omega =  \, \partial \overline{\partial}
f_{\omega}.$$

Given $Z \in \mathfrak{h}(M)$, define:
$$F(Z):= \int_{M} Z (f_{\omega}) \, \omega ^n$$

In~\cite{f} it is proven that the integral on the right hand side
does not depend on the Kahler form $\omega$ in the class $c_1(M)$,
and that $$F: \mathfrak{h}(M) \to \C $$ is a homomorphism of Lie
algebras.

\end{df}

\begin{rmk}
The vanishing of this invariant is a necessary condition for the
existence of a Kahler-Einstein metric (\cite{f}). In the case of
smooth toric manifolds this is also sufficient (\cite{wang zhu}).
\end{rmk}

\bs

Even though $F$ is defined for vector fields and $I$ - for loops,
when restricted to the group

$$K:=Iso_{\,0}(M):= (Ham(M,\omega) \bigcap Aut(M,J))_{0}$$

on which both are defined, such vector fields and loops are
essentially equivalent, since $$\pi_1(K) \otimes \R \cong
Lie(K)/[Lie(K),Lie(K)]$$

So that, by taking duals, one gets an equivalence between
homomorphisms of abelian groups $\pi_1(K) \rightarrow \R$ and
homomorphisms of Lie algebras $Lie(K) \rightarrow \R$.

\bs

 One then has the following equality:

\begin{thm}\label{fano equality}

Let $(M, J, \omega)$ be a compact Fano manifold.

Then $$F=- \mVol \, I$$ when restricted to $K=Iso_0(M,J,\omega)$.

The equality is understood via the isomorphism $Hom(\pi_1(K),\R)
\cong Hom(Lie(K),\R)$.
\end{thm}

There are lots of generalizations of the Futaki invariants, with
perhaps the most inclusive family being the one introduced
in~\cite{f2004}. We present the original definition in a slightly
rewritten form, for the case of the frame bundle of the holomorphic
tangent bundle of a compact Kahler manifold.

\begin{df}
{\it(Generalized Futaki Invariants)}

Let $(M,\omega,J)$ be a compact Kahler manifold. Denote by
$\mathfrak{h}_0(M)$ the space of holomorphic vector fields in
$\mathfrak{h}(M)$ that have a zero. It's known that
$$\mathfrak{h}_0(M) = \{Z \in \mathfrak{h}(M) | \exists! f_Z \in
C^{\infty}(M,\C) \text{ s.t. } i_Z \omega = - \bar{\partial}f_Z \
\& \ \int_M f_Z \omega^n = 0 \}$$ (e.g.~\cite{lebrun simanca})

Consider the frame bundle $Fr$ of $T^{(1,0)}(M)$. This is a
principal holomorphic $GL(n,\C)$ bundle over $M$. The automorphism
group $Aut(M,J)$ acts on $Fr$ commuting with the action of
$GL(n,\C)$.

The invariant polynomials on $Lie(GL(n,\C))$ are generated by
elementary symmetric polynomials in the eigenvalues, which are
taken with suitable coefficients to correspond to Chern classes.
Denote them by $c_1,...,c_n$.

Take $Z \in \mathfrak{h}_0$. Since $Aut(M,J)$ acts on $Fr$ by
bundle maps, $Z$ lifts to a holomorphic vector field $\widehat{Z}$
on $Fr$.

Set $w_Z:=\omega + f_Z$. This is a form of mixed degree.

Choose a $(1,0)$ connection form $\theta$ on $Fr$. Let $\Theta$ be
the corresponding curvature form.

Take a polynomial $(c_1)^{i_1} \cdot ... \cdot (c_n)^{i_n}$ of
degree $k \leq n$, and take ${w_Z}^l$, s.t. $k+l = n$.

Denote $\alpha := (c_1)^{i_1} \cdot ... \cdot (c_n)^{i_n} w^l$.

Then $$F_{\alpha}(Z):= \int_M (c_1)^{i_1} \cdot ... \cdot
(c_n)^{i_n} (\theta(\widehat{Z})+ \Theta){w_Z}^l. $$

It is shown in~\cite{f2004} that $F_{\alpha}(Z)$ is independent of
the choice of the $(1,0)$ connection form $\theta$, and defines a
homomorphism of Lie algebras:

$$F_{\alpha}:\mathfrak{h}_0(M) \to \C$$
\end{df}

\begin{rmk}
$F_{\alpha}$ also depends only on the cohomology class of the
Kahler form $\omega$.
\end{rmk}

\begin{rmk}
The invariant $F_{c_1 w^n}$, also known as the Bando-Calabi-Futaki
character, is an obstruction to the existence of a constant scalar
curvature Kahler metric with the Kahler form in the given cohomology
class. It also has a definition similar to Definition~\ref{Futaki
basic} (peruse~\cite{f2004} and the references therein)
\end{rmk}

\begin{rmk}
Following~\cite{fm,fm2}, one can relate the invariants $F_{\alpha}$
to a type of equivariant cohomology.
\end{rmk}

 Theorem~\ref{fano equality} then generalizes to the following:

\begin{thm}\label{generalized equality}

Let $(M, J, \omega)$ be a compact Kahler manifold.

Then $$F_{\alpha}=I_{\alpha}$$ when restricted to
$K=Iso_0(M,J,\omega)$.

The equality is understood via the isomorphism $Hom(\pi_1(K),\R)
\cong Hom(Lie(K),\R)$.
\end{thm}

\begin{rmk}
Theorem~\ref{fano equality} follows from Theorem~\ref{generalized
equality} for the following reason: in~\cite{fm} (Proposition 2.3)
it is proven that for Fano Kahler manifolds, $F_{(c_1)^{n+1}}= (n+1)
F$. Also from Computation~\ref{proportional}, one has
$I_{(c_1)^{n+1}}= - (n+1) \,\mVol \cdot I$. So $F_{(c_1)^{n+1}}=
I_{(c_1)^{n+1}}$ means $F = - \mVol \, I$.
\end{rmk}

\begin{rmk} \label{phenomenon}
Theorem~\ref{generalized equality} yields the equality $-I_{c_1
u^n}/\mVol= I = I_{c_n u}/{Euler(M)}$ on $K$ for Fano manifolds.

Indeed, talking in the language of symplectic fibrations, $u = c_1 +
I \cdot \pi^*(a)$ for Fano manifolds, where $a$ is the generator of
$H^2(S^2,\R)$, by~\cite{polterovich loops}. So $ I_{c_n u}(\gamma) =
I_{c_n c_1} + I <c_n \pi^*(a), [P_{\gamma}]> = I_{c_n c_1} + I \cdot
Euler(M)$. By Theorem~\ref{generalized equality}, $I_{c_n c_1} =
F_{c_n c_1}$ on $K$. And the latter vanishes on its domain of
definition (\cite{fm,fm2,f book}). So $ I_{c_n u}(\gamma) = I \cdot
Euler(M)$, and the rightmost equality is proven. The leftmost
equality follows from Computation~\ref{proportional}.

\end{rmk}

\begin{rmk}
Theorem~\ref{generalized equality} also lets one recover the toric
computations of $I_{c_1 u^n}$ in~\cite{vinew} from the computations
of $F_{c_1 w^n}$ in~\cite{nakagawa1,nakagawa2}.
\end{rmk}

\subsection{Example: the case of toric loops}\label{toric computations}

Although the computation of $I$ on all loops on toric manifolds yet
remains an open issue, one can compute its restriction to those
Hamiltonian loops that come from the torus action. We'll address
these as "toric loops".

A homomorphism on $\pi_1(T)$ can be considered as a vector in
$Lie(T)^*$. The corresponding vectors are most conveniently
expressed through the following notion of \emph{barycenters}
$\{B_k\}_{0 \leq k \leq n}$ of Delzant polytopes:

\begin{df}
{\it (k-dimensional measure on $\Delta$)} As the faces of the
polytope are rational, there is a lattice on each induced from the
integer lattice of the ambient space. This lattice defines up to a
multiplicative constant a measure on the face. The constant can be
normalized in such a way that any fundamental parallelotope has
measure 1.

The \emph{k-dimensional measure} is the sum of such measures over
all $k$-dimensional faces.

\end{df}

\begin{df}
{\it (The k-th barycenter)}

 $B_k := \text{barycenter of the k-dimensional measure.}$
\end{df}

\begin{rmk}
Note that the k-dimensional measure is the push forward by the
moment map of $(\omega^k)/{k!}$ restricted to the corresponding
$2k$ dimensional symplectic invariant subspaces of $M$.
\end{rmk}

Consider a Delzant Polytope with vertices $\{P_k\}_{1 \leq K \leq
V}$, and faces $\{F_j\}_{1 \leq j \leq F}$ given by primitive
"normals" $\{l_j\}_{1 \leq j \leq F}$ in the integer lattice of
the dual space. So that $$\Delta=\cap_{j=0}^F \{l_j \leq
\kappa_j\}$$ for some support numbers
$\overrightarrow{\kappa}=(\kappa_1,...,\kappa_F)$.

For monotone toric manifolds, $I$ was computed in~\cite{rigid}. The
result is expressed as follows in our notations:
$$I = - B_0 + B_n$$.

For general toric manifolds, $I_{c_1 u^n}$ was computed
in~\cite{vinew}. The result is expressed as follows in our
notations:
$$I_{c_1 u^n} = -n! \text{Vol}_{(n-1)}(\Delta)(B_{n-1} - B_n)$$

\begin{rmk}
Comparing the formulas and using Computation~\ref{proportional}, one
obtains $$B_0-B_n = - C_{\Delta} (B_{n-1} - B_n)$$ where $C_{
\Delta}=\frac{\kappa
\text{Vol}_{n-1}(\Delta)}{\text{Vol}_n(\Delta)}$ is a positive
constant. Which means, curiously enough, that the three barycenters
$B_0,B_{n-1}$ and $B_n$ are collinear for mononotone toric
manifolds, and are either all distinct or all equal.

\end{rmk}

\begin{rmk}
Another way to get the collinearity of $B_0, B_{n-1},B_n$ in the
toric Fano case is to compare the two toric computations of $F$
in~\cite{mabuchi}, and~\cite{donaldson}.
\end{rmk}

What follows is a uniform computation of all $\displaystyle
\{I_{c_L u^{n+1-L}}\}_{0 \leq L \leq n}$.

\begin{thm}\label{bary I}

$I_{c_L u^{n+1-L}} / (n-L)! \text{Vol}_{(n-L)}(\Delta)=
-(n+1-L)(B_{(n-L)} - B_n)$.

\end{thm}

\begin{rmk}
In particular, $I_{c_n u}/\text{Vol}_0(\Delta)= -B_0 + B_n$, on the
torus.
\end{rmk}

\begin{rmk}\label{phenomenon toric} An explanation of the collinearity phenomenon in the toric Fano case
follows from this equality and Remark~\ref{phenomenon}, via
$\text{Vol}_0(\Delta)$ = number of the vertices in the polytope =
$Euler(M)$.

\end{rmk}

\bs

This computation has a corollary related to a result in~\cite{poly
with ml}:

Consider a symplectic toric manifold, and a Hamiltonian loop
$\gamma$ coming from the toric action. This loop corresponds to a
point $l$ in the integer lattice in $Lie(T)$.

Let the moment polytope be
$\Delta=\Delta(\overrightarrow{\kappa}(0))=\cap_{j=0}^F \{l_j \leq
\kappa_j(0)\}$ for some support numbers
$\overrightarrow{\kappa}(0)=(\kappa_1(0),...,\kappa_F(0))$.

Denote by $C$ the chamber of all support numbers
$\overrightarrow{\kappa}$ for which the polytope
$\Delta(\overrightarrow{\kappa}) $ is analogous to the original one
(\cite{poly with ml}). One can think of continuous deformations of
$\Delta$ in the space of Delzant polytopes with the given conormals.

In~\cite{poly with ml} it is proven that if  $\gamma$ is
contractible in $Ham$, then the function
$f(\overrightarrow{\kappa}):=<B_n(\overrightarrow{\kappa}), l>$ is a
linear function of $\overrightarrow{\kappa}$ with integer
coefficients (that is, the restriction of such a function to $C$).

In other words, $l$ considered as a function on $Lie(T)^*$ is
\emph{mass-linear} with integer coefficients.

A lemma based on Moser's homotopy method is used, saying that if
such a toric loop is contractible in $Ham$ for the original
polytope, then it will be contractible for all
$\overrightarrow{\kappa}$ in an open neighbourhood $U$ of the
original $\overrightarrow{\kappa}(0)$.

 Note that as $<B_n(\overrightarrow{\kappa}), l>$ is a priori a
 rational function of $\overrightarrow{\kappa}$, it's enough to show
 linearity on such an open set.

\bs

 Theorem~\ref{bary I} lets one prove a related result in the
following manner:

\begin{cor} \label{mass linear}
Assume a loop $\gamma$ in $Ham$ coming from the toric action is
contractible. Let $l$ be the corresponding point in the integer
lattice in $Lie(T)$.

\bs

 Then $$<B_n(\overrightarrow{\kappa}), l> =
 <B_0(\overrightarrow{\kappa}),l>$$

 (which is a linear function of $\overrightarrow{\kappa}$ with coefficients in $\frac{1}{Vol_0}\Z$) \bs

\bs

and $$<B_n(\overrightarrow{\kappa}), l> =
 <B_k(\overrightarrow{\kappa}),l>$$

 for all other $k$, as well.
\end{cor}

\begin{pf}
If $\gamma$ is contractible then $I_{c_n u }(\gamma)=0$ (for all
$\displaystyle \overrightarrow{\kappa}$ in $U$). Which, by
Theorem~\ref{bary I}, means that
$<B_n(\overrightarrow{\kappa})-B_0(\overrightarrow{\kappa}),l> =0$
for all $\displaystyle \overrightarrow{\kappa} \in U$. Hence, as
both are rational functions of $\overrightarrow{\kappa}$, the
equality holds for all $\displaystyle \overrightarrow{\kappa}$ in
$C$.
 The second part is obtained similarly, by applying Theorem~\ref{bary I}
to $I_{c_{n-k}u^{k+1}}$.
\end{pf}

\begin{rmk}
An alternative way to state this corollary is that if $\gamma$ in
$Ham$ coming from $l \in \pi_1(T) \subset Lie(T)$ is contractible,
then $l$ is perpendicular to the affine span of
$\{B_k(\overrightarrow{\kappa})\}_{0 \leq k \leq n}$ for all
$\overrightarrow{\kappa}$ in $C$.

Also, since under the contractibility condition
$<B_n(\overrightarrow{\kappa}), l>$ is linear with integer
coefficients (\cite{poly with ml}), it follows that all
$<B_k(\overrightarrow{\kappa}), l>$ are linear with integer
coefficients. This could be applied to
$<B_0(\overrightarrow{\kappa}),l>$, which has coefficients in
$\frac{1}{Vol_0}\Z$, a-priori, to give lower bounds on the orders of
torsion elements of $\pi_1(Ham)$ represented by toric loops.
\end{rmk}

\begin{rmk}
In~\cite{vinewnew}, for $\C P^n$ bundles over $\C P^1$ and for the
blowup $Bl_1(\C P^n)$ of $\C P^n$ at one point, a similar result was
obtained by using the invariant $I_{c_1 u^n}$.
\end{rmk}

\subsection{A map of the rest of the article}

In section~\ref{proofs} the theorems are proven in the order of
their appearance. A stand-alone differential geometric proof of
Theorem~\ref{fano equality} is provided, as it is interestingly
similar to the proof of Theorem~\ref{equality of m and cw}.

In section~\ref{discussion} several related questions are posed.

\section{Proofs}\label{proofs}

\subsection{Proof of Theorem~\ref{equality of m and cw}}

Let $\{\widehat{\phi_t}\}_{t \in S^1}$ be a loop in $Q$, and
$\{\phi_t\}_{t \in S^1}$ be the corresponding "downstairs" loop in
$H$. Let $F_t$ be the normalized Hamiltonian function for
$\{\phi_t\}_{t \in S^1}$.

Note that $-h_t$ is a (non-normalized) Hamiltonian for
$\{\phi_t\}_{t \in S^1}$. (the minus sign follows from the fact that
the tautological line bundle has Chern class $-\alpha$, where
$\alpha$ is the generator of the cohomology of $\C P^n$)

As both $-h_t$ and $F_t$ are both hamiltonians for the same loop
in $H$, they differ by a constant dependent on time.

$$ h_t = F_t + c(t) $$

or, for future reference,

$$  -F_t = -h_t + c(t) $$

Then by integrating over M, we get:

$$ Vol(M,\omega^n) \cdot \int_0^1  c(t) dt = \int_0^1 dt \int_M h_t \omega^n $$

But the right hand side equals $cw(\{\widehat{\phi_t}\})$,
therefore

$$ \int_0^1  c(t) dt =  \frac{1}{Vol(M,\omega^n)} \cdot  cw(\{\widehat{\phi_t}\})$$

\bs

For a point $1 \in P_y$ for $y \in \C P^n$ consider the path
$\{\widehat{\phi_t} 1\}_{t \in S^1}$ in $P=S^{2n+1}$. It has a
unique, up to homotopy with fixed boundary, filling disk
$\widehat{D_0}$, since $\pi_1(S^{2n+1})=0$ and $\pi_2(S^{2n+1})=0$.
And to this disk there corresponds a canonical filling disk $D_0 :=
p \circ \widehat{D_0} $ of the "downstairs" path $\{\phi_t y\}_{t
\in S^1}$.

 From this point, the theorem follows from the following two
lemmas:

\begin{lma}\label{maslovs}
Consider a trajectory $\displaystyle m(\{ {\Phi_t(1)}\}_{t \in
S^1})$ for a point $1 \in P_y \subset SP_y \cong \C ^*$ in the
fiber over y.

Then, for the canonical filling disk $D_0$,

 $\displaystyle m(\{ {\Phi_t}_{*_1}\}_{t \in S^1}) = m_{D_0}(\{{\phi_t}_{*_p}\}) $

\end{lma}

\begin{lma}\label{maslov I and cw }
For the canonical filling disk $D_0$ it is true that:

$\frac{1}{2(n+1)}m_{D_0}(\{{\phi_t}_{*_p}\})=  - I(\{{\phi_t}\}) +
\int_0^1 c(\tau) d \tau $

\end{lma}

\bigskip

These lemmas yield,

$$ \frac{1}{2(n+1)} m(\{{\Phi_t}_{*_1}\}) = - I(\{{\phi_t}\}) +  \int_0^1 c(\tau) d \tau $$

But since,  $I \equiv 0$, by Seidel's argument, we'll have

$$ \frac{1}{2(n+1)} m(\{{\Phi_t}_{*_1}\}) = \int_0^1  c(t) dt = \frac{1}{Vol(M,\omega^n)} \cdot
cw(\{\widehat{\phi_t}\}),
$$

So that

$$\displaystyle \frac{1}{2(n+1)} \mu(\{\widehat{\phi_t}\}) =
\frac{1}{Vol(M,\omega^n)} \cdot cw(\{\widehat{\phi_t}\})$$

and we're done.

\bs \bs

We proceed to prove the lemmas:

\bs

 \emph{Proof of Lemma~\ref{maslovs}}

Since $\widehat{\phi_t}$ is a quantomorphism, it preserves the
vertical and the horizontal subbundles of $TP$.

Therefore $\Phi_t$ preserves the corresponding vertical and
horizontal subbundles of $TL^{\times}$.

So, $m(\{{\Phi_t}_{*_1}\}) = m_{\widehat{D_0}}(\{{\Phi_t}_{*_1}\}) =
m_{\widehat{D_0}}(\{{\Phi_t}_{*_1}|_{Hor}\}) +
m_{\widehat{D_0}}(\{{\Phi_t}_{*_1}|_{Vert}\})$

But $m_{\widehat{D_0}}(\{{\Phi_t}_{*_1}|_{Hor}\}) =
m_{D_0}(\{{\phi_t}_{*_y}\})$, since $Hor \cong p^* TM$ as
symplectic vector bundles;

And $m_{\widehat{D_0}}(\{{\Phi_t}_{*_1}|_{Vert}\})=0$, since
${\Phi_t}_* : Hor_1 \rightarrow Hor_{\Phi_t 1}$ is just equal to the
parallel translation map: $\Gamma_{\{\Phi_s 1\}_{s=0}^1}: Hor_1
\rightarrow Hor_{\Phi_t 1}$ since both preserve the connection
1-form, and the fiber is 1 dimensional.

 And since the connection on $\C ^{n+1} \setminus \{0\}$ is
trivial, $m_{\widehat{D_0}}(\{{\Phi_t}_{*_1}|_{Vert}\}) =
m({\{Id\}_{t \in S^1 }})$.

\bs

\emph{Proof of Lemma~\ref{maslov I and cw }}

By definition, $I(\{{\phi_t}\})= \int_{D_0} \omega - \int_0^1
F_t(\phi_t y)dt - {\frac{1}{n+1}} \cdot {\frac{1}{2}} \cdot
m_{D_0} (\{{\phi_t}_{*_y}\}) $

But $-\int_0^1 F_t(\phi_t y)dt = - \int_0^1 h_t(\widehat{\phi_t}
1) + \int_0^1 c(t) dt $

Substituting, and noting that: $\int_{D_0} \omega - \int_0^1
h_t(\widehat{\phi_t} 1) = 0$, by the Stokes formula (indeed,
$\int_{D_0} \omega - \int_0^1 h_t(\widehat{\phi_t} 1) =
\int_{\widehat{D_0}} p^* \omega - \int_{\{\widehat{\phi}_t 1\}_{t
\in S^1}} \alpha = \int_{\widehat{D_0}} d \alpha - \int_{\partial
\widehat{D_0}} \alpha$ since $p^*\omega = d \alpha$, and
$\{\widehat{\phi}_t 1\}_{t \in S^1 } =
\partial \widehat{D_0}$ ), we get

 $I(\{{\phi_t}\}) - \int_0^1 c(t) dt =  - {\frac{1}{2(n+1)}}
\cdot m_{D_0} (\{{\phi_t}_{*_y}\}) $

So $\frac{1}{2(n+1)} m_{D_0} (\{{\phi_t}_{*_y}\}) = -
I(\{{\phi_t}\}) + \int_0^1 c(\tau) d \tau $

\subsection{Proof of Corollary~\ref{m cw and pi_1}} \label{proof of cor to m cw}

{\bf Claim:} There exists a homomorphism $\alpha: \pi_1(H)
\rightarrow {\Z}/{{(n+1)}\Z}$, such that\\
\;\;1. $Ker (\alpha) = Ker(cw|_{\pi_1(Q)})$ and \\
\;\;2. on the element $\theta$ represented by the $S^1$ action of
the rotation of the first homogenous coordinate, it takes value $1
\in {\Z}/{{(n+1)}\Z}$

\bs

Note that property $2$ means that $\pi_1(H) \supset \; <\theta>
\cong {\Z}/{{(n+1)}\Z}$ and $\alpha|_{<\theta>}: <\theta>
\rightarrow {\Z}/{{(n+1)}\Z}$ is an isomorphism.

 The theorem now follows from the claim, as $\pi_1(H)$ is abelian.

\bs

{\em Proof of Claim:}

 Consider a loop $\gamma = \{{\phi_t}\}_{t
\in S^1}$ based at $Id$ in $H$, and lift it to a path
$\widehat{\gamma}= \{\widehat{\phi_t}\}$ in $Q$ (by the "canonical
lifting") using the mean-normalized Hamiltonian. Obviously, this
path represents an element of $Ker(cw)$. Note that if it closes up
to a loop, then it represents one in $Ker(cw|_{\pi_1(Q)})$.

Since it's a lift of a loop downstairs, we can close it up by a path
$\delta:= \{R_t\}_{0 \leq t \leq \alpha}$, $ \alpha \in {\R}/{\Z}$
of ("Reeb") rotations of the fibers, to get a loop
$\widehat{\gamma}*\delta$ in $Q$.

Note that $\gamma \mapsto \alpha$ gives a homomorphism $\pi_1(H)
\rightarrow {\R}/{\Z}$. We contend that this homomorphism actually
takes values in ${\frac{1}{(n+1)}\Z}/{\Z} \cong {\Z}/{{(n+1)}\Z}.$

Indeed, choose a representative $0 \leq \alpha < 1$ for $\alpha$.
Then, $\frac{1}{2(n+1)}m(\gamma)= \frac{1}{\mVol}
cw(\widehat{\gamma}) = \alpha$. So, $\alpha \in \frac{1}{n+1} \Z$
(since the Maslov index is always even).

Property $1$ now follows from this computation.

Property $2$ is a straightforward computation. One obtains that
the representative $0 \leq \alpha < 1$ equals $\frac{1}{n+1}$,
that is $1 \in {\Z}/{{(n+1)}\Z}$.

\bs

\subsection{Proof of Theorem~\ref{fano equality}}

First, note that it is enough to prove the equality on $S^1$
subgroups $\Phi_t$ of $K$, since $K$ is a compact Lie group.

The proof is based on the fact, that the Calabi-Yau theorem gives
$(M,\omega)$ a canonical prequantization.

The prequantization is built as follows: Let
$L=\bigwedge^nT^{(1,0)}M^*$ be the canonical line bundle on M.
 Let $\eta$ be the Calabi-Yau form with $Ric(\eta)=\omega$.

 It gives us a Hermitian metric on $T^{(1,0)}M$. This, induces, in turn a
Hermitian metric $\rho$ on $L$. To the pair $(L,\rho)$ there
corresponds a unique complex connection $\nabla_{Ch}$ which is
compatible with the metric and with the structure of a holomorphic
bundle (it is called the Chern connection of $(L,\rho)$). Denote
by $\bar{\alpha}$ the corresponding connection one-form alpha with
values in $\mathbb{C}$.

Now the prequantization $(P,\alpha)$ is the principal $S^1$-bundle
$P$ of $\rho$-unit vectors in $L$, together with the
\textit{$\mathbb{R}$-valued} connection one-form
 $\alpha=-i\bar{\alpha}|_P$.

As $\Phi_t$ are automorphisms of all the structures involved, they
lift canonically to automorphisms $\hat{\Phi_t}$ of $(P,\alpha)$.

In detail the lift acts as follows:

\begin{equation*}
      \begin{array}{clcr}
         P & \rightarrow & P \\
         (x,p) &\mapsto   & (\Phi_t(x),((\Phi_{t_{\ast_x}})^{-1})^{\ast}p) \\
         x \in M, p \in P_x &  &
      \end{array}
   \end{equation*}

To the loop of automorphisms $\hat{\Phi_t}$ there corresponds a
contact Hamiltonian
$\hat{h}=\alpha(\frac{d}{dt}|_{t=0}\hat{\Phi_t})$. This is an $S^1$
invariant function, and therefore can be considered as a function
$h$ on $M$.

\bigskip

\begin{df}
 {\it (divergence of $Z \in \mathfrak{h}(M)$ w. r. t. a Kahler form
$\eta$ )}
\\For a holomorphic vector field $Z \in \mathfrak{h}(M)$ define
$div_{\eta}(Z)$ to be the unique function $\psi \in
C^\infty(M,\C)$, such that $d(i_Z\eta^n)=\psi \eta^n$.
\end{df}

The proof is composed of a proposition from~\cite{fm}, and two
lemmas:

\begin{lma} \label{Futaki and div}
(A.Futaki, S.Morita~\cite{fm}) $F(Z)= \frac{1}{i} \int_{M}
div_\eta(Z)\,\omega ^n$ where $\eta$ is the Calabi-Yau form of
$\omega$ that is defined uniquely by the condition
$Ric(\eta)=\omega$.
\end{lma}

\begin{lma} \label{div and h}
Let h be the contact Hamiltonian introduced earlier. Then $ih= -div
(Z)$.
\end{lma}

\begin{lma} \label{h and I}
For the same contact Hamiltonian h,
$$\int_{M} h \, \omega ^n = \mVol \cdot I(\gamma)
$$
Where $\gamma$ is the forementioned loop of automorphisms
$\{\Phi_t\}, t \in \mathbb{R} / \mathbb{Z}$.
\end{lma}

\bs

 Indeed, given these lemmas one has:

$$F(Z)= \frac{1}{i} \int_{M}
div_\eta(Z)\,\omega ^n = \frac{1}{i} \int_{M} -ih\,\omega ^n = -
\mVol \cdot I(\gamma).$$

\bs

 We now go on to prove the lemmas:

\begin{proof}[Proof of Lemma~\ref{h and I}]

Let $h$ be the contact Hamiltonian introduced earlier. Denote by
$H$ the normalized Hamiltonian of the flow $\Phi_t$.

Then $h = H + c$ ,where $c \in \R$ is a constant.

Therefore

$$  \int_{M} h  \omega ^n = \mVol \cdot c.$$

{\bf Claim:} $c = I(\gamma)$.

The lemma follows from the claim by substituting into the last
formula. Indeed, we obtain $$\int_{M} h \, \omega ^n = \mVol \cdot
I(\gamma)$$.

{\em Proof of Claim:}

Indeed, $c = h(p) - H(p)$  for any $p \in M$.

As M is compact, one can choose p to be a critical point of H. This
point is a fixed point of the flow. Therefore computing I w.r.t.
this point and the trivial disk,
 $$I(\gamma)= - H(p) - Maslov/2(\gamma_*)$$
where $\gamma_*(t) = \Phi_{t_{*_x}}$ is the linearized loop of the
flow at this point.

Yet $-h(p)=$ (speed of rotation of $det(\gamma_*(t))$ in the fibre
over $p$). And this, as the speed is constant, in turn equals to
($\#$ of full turns of $det(\gamma_*(t)))$ which, by definition of
the Maslov index, is $Maslov/2(\gamma_*(t))$.

Therefore, $$I(\gamma)= - H(p) + h(p).$$

\end{proof}

\begin{proof}[Proof of Lemma~\ref{div and h}]

We use a formula proved in~\cite{goldstein} (Proposition 2.1.2, p.
237):
\bigskip\\
 Let $s$ be a section of the canonical line bundle $L$, and
$X$ - a (real) holomorphic vector field on $M$. Then,

$$ \mathcal{L}_X s = \nabla_X s + div(X)\cdot s$$
\\
where $div(X)$ is defined as follows:
\\
$$div(X)_p:=trace_{\C}(V \rightarrow \nabla_V X ).$$
\\
(note that $V \rightarrow \nabla_V X$ is a $J$-linear operator
$T_pM \rightarrow T_pM$.) \bigskip
 \\Locally, we can choose a flat section $s$ of unit
length. This is possible, as the connection preserves the
Hermitian metric.

For a flat section, the equation reduces to:

$$ \mathcal{L}_X s = div(X)\cdot s$$

On the other hand:
$$\displaystyle  \mathcal{L}_X s = \frac{d}{dt}|_{t=0}\Phi_t^* s $$

and as $\Phi_t^* s$ is also a flat section of unit length,

$$\Phi_t ^* s = e^{ia(t)} \cdot s$$  where $a(t)$ is a function of
the time only (because of the flatness condition). Moreover $-a'(0)=
h$ is the contact Hamiltonian.

Therefore,

$$\displaystyle  \mathcal{L}_X s = \frac{d}{dt}|_{t=0}e^{ia(t)} \cdot s =$$
$$\bigskip = ia'(0) \cdot s = - i \, h \cdot s$$

Comparing the two computations of $\mathcal{L}_X s$, we obtain
$$ div(X)= -i \, h$$

The last remark is that for $Z:=(X-iJX)/2 \in \mathfrak{h}(M)$
$$div_\eta(Z)=div(X).$$
(by the same J-linearity; see e.g.~\cite{kob nom})
\end{proof}


\subsection{Proof of Theorem~\ref{generalized equality}}

The key lemma of the proof is the following simple fact:

\begin{lma} \label{basic}
Let $(M,\omega,J)$ be a compact Kahler manifold. Let $f\in
C^{\infty} (M,\R)$ be a real valued smooth function on $M$. Then

$$i_X\omega = -df \  \Leftrightarrow \ i_Z\omega = -
\bar{\partial}f$$ where $Z=(X-iJX)/2$.

\end{lma}

As in the proof of Theorem~\ref{fano equality}, it's enough to prove
the given equality on $S^1$ subgroups of $K$.

Given such an $S^1$ subgroup $\gamma$ of $K$ one has the
corresponding Hamiltonian fibration $P$ over $\C P^1$. According
to~\cite{lmp} (Remark 3.C), this bundle is just the restriction to
$\C P^1$ of the universal bundle $M_{S^1}:=M\times_{S^1}S^{\infty}
\to \C P^{\infty}$. And so, the cohomology $H^*(P,\R)$ is just
$H_{S^1}(M,\R)\otimes_{\R[z]}{\R[z]/z^2\R[z]}$.

Also, the vertical Chern classes of $P$ are nothing but the
restrictions equivariant Chern classes of the tangent bundle.

Denote by $f$ the normalized Hamiltonian function of the $S^1$
action, and by $X$ the corresponding vector field.

Consider now the Cartan model for equivariant cohomology:

The coupling class is represented by $u=\omega + z f$.

Given an $S^1$-equivariant connection form $\theta$ on the frame
bundle of the complex vector bundle $TM$ and its curvature $\Theta$,
the equivariant Chern classes are represented by $c_r(z
\theta(\widehat{X})+\Theta)$, where $c_r$ is the elementary
symmetric polynomial of degree $r$, and $\widehat{X}$ is the lift of
$X$ to the frame bundle (proven e.g. in~\cite{bt}).

 Therefore the value of the generalized mixed action-Maslov
invariant $I_{\alpha}$ on $\gamma$ is given by
$$I_{\alpha}(\gamma)=\int_M (c_1)^{i_1} \cdot ... \cdot (c_n)^{i_n}
(\theta(\widehat{X})+\Theta)) (\omega + f)^j.$$

By invoking Lemma~\ref{basic} and choosing the connection to be of
type $(1,0)$ this now equals to $F_{\alpha}(Z)$ with $Z =
(X-iJX)/2$.

\subsection{Proof of Theorem~\ref{bary I}}

Take a toric $S^1$ action $\gamma$ with mean normalized Hamiltonian
$f$.

Using the language of Hamiltonian fibrations, one has, by definition
$$I_{c_L u^{n+1-L}}(\gamma)= \int_{P_{\gamma}} c_L u^{n+1-L} =
<PD_P(c_L), u^{n+1-L}>$$

Denote by $\Delta_{n-L}$ the formal sum of the faces of $\Delta$ of
$dim= n-L$. And let $N=m^{-1}(\Delta_{n-L})$ be the formal sum of
the preimages of these faces by the moment map (note that these
preimages are all T-invariant).

Let $F_{n-L}$ be the chain obtained from $N$ in the same way as
$P_{\gamma}$ is obtained from $M$. Then $PD_P(c_L)$ is represented
by $F_{n-L}$. This follows from~\cite{kj}, the fact that the first
Chern class of a holomorphic line bundle is the Poincare dual of the
corresponding divisor, and choosing an equivariant meromorphic
section for each of the relevant holomorphic line bundles.

According to~\cite{polterovich loops}, the coupling class is
represented by the form:

\bs

\centerline{\{$\omega$ on $M \times D_{+}$; $\omega + d (\psi(r)
f(x) dt))$ on $M \times D_{-}$ \}}

\bs

 where $D_{+}$ and $D_{-}$ are two
disks (from which the sphere is glued),
 $r$ and $t$ are the
radial and the angular coordinates on the disk $D_{-}$, $x$ denotes
a point on $M$, and $\psi$ is a function that vanishes near $0$ and
equals to $1$ near $1$.

\bs

So that $u^{n+1-L}$ restricted to $F_{n-L}$, equals to

\bs

\centerline{\{0 on $N \times D_{+}$; $(n+1-L){\omega}^{n-L}\psi(r)'
f(x) dr dt$ on $N \times D_{-}$ \}.}

\bs

Hence $\int_{F_{n-L}} u^{n+1-L}$ = $(n+1-L)\int_{N}\,f\,
\omega^{(n-L)}$, so that:

$$<PD_P(c_L), u^{n+1-L}>= (n+1-L)\int_{N}\,f\,
\omega^{(n-L)}$$

Yet, the right hand side is just equal to
$-(n+1-L)!Vol_{(n-L)}(\Delta)(B_{n-L} - B_n)$ evalutated on the
point $l$ in the integer lattice corresponding to $\gamma$
 (the "-" sign
stems from the difference in sign conventions for hamiltonians and
moment maps).

And therefore,

$$I_{c_L u^{n+1-L}} / Vol_{(n-L)}(\Delta)= -(n+1-L)!(B_{(n-L)} -
B_n)$$

\bs

as required.

\section{Discussion and Questions}\label{discussion}

The following are questions related to the subject of this note.

\begin{enumerate}

\item

 It would be interesting to make further comparisons of the
computations in subsection~\ref{toric computations} to the results
of~\cite{poly with ml}. In particular, Theorem~\ref{mass linear}
raises the following inverse question:

\begin{qtn}

Assume that a loop $\gamma$ in $Ham$ coming from the toric action
is such that all $I_{\alpha}(\gamma)=0$ for all
$\overrightarrow{\kappa}$ in the chamber. Does it follow that
$\gamma$ is contractible?

\end{qtn}

\item

 It was shown in~\cite{py qm} that for monotone symplectic
manifolds the mixed Action-Maslov invariant extends to a homogenous
quasimorphism on the universal cover of the Hamiltonian group. Then
it was asked by L. Polterovich (\cite{poterovich}) whether one could
extend $\displaystyle I_{c_1 \, u^n}$ to such a quasimorphism in the
non-monotone case. As a first step, it would be interesting to check
this for Kahler manifolds of {\em constant scalar curvature}.

\item

 It would be interesting to investigate the collinearity phenomenon
 of barycenters of toric Fano polytopes, beyond the triple $B_0, B_n, B_{n-1}$.
 Are all the barycenters collinear? Are they collinear in triples $B_L, B_n, B_{n-L-1}$? Is this related to a certain duality?
 The first nontrivial case for investigation is in complex dimension 4.
 For example, taking the Ostrover-Tyomkin polytope (\cite{yaron tyomkin}, section 5), it
would be interesting to check whether $B_1,B_2,B_4$ are collinear.
In general, following Remarks~\ref{phenomenon} and~\ref{phenomenon
toric}, further vanishings of Futaki invariants (\cite{f book})
might be of use in addressing these questions.
\item

For toric Fano  manifolds $(M,\omega,J)$, whenever $I=0$ on toric
loops, the Futaki invariant $F$ vanishes. This is equivalent,
by~\cite{wang zhu}, to the existence of a Kahler-Einstein metric.
However not all Hamiltonian loops are necessarily toric, so it's
interesting to answer

\begin{qtn}
Do Kahler-Einstein toric Fano manifolds $(M,\omega,J)$ have $I=0$
identically?
\end{qtn}

Here an extension of Seidel's argument from \cite{calabi qm},
section 4.3 could be of essence.

It may also be interesting to investigate the same question for
general Kahler-Einstein manifolds.
\end{enumerate}

\section*{Acknowledgements}

I would like to thank my advisor Leonid Polterovich for his
comprehensive guidance and apt help, and for giving my proofs a
refinedly elegant touch. I would also like to thank Michael Entov
and Leonid Polterovich for introducing me to the question of
Theorem~\ref{fano equality}, which was suggested by Chris Woodward,
and for their help with it. Many thanks to Gabi Ben-Simon for
introducing me to the conjecture of~\cite{givental}, to Frol
Zapolsky for reintroducing it to me, and to them both for
stimulating conversations. I would like to thank Joseph Bernstein
for teaching me a lot, about Differential Geometry in particular,
and for useful comments. I would like to thank Dusa McDuff for
useful comments on an earlier version of this paper. I thank the
organizers of the UK-Japan Winter School 2008 for enabling me to
present a preliminary version of Theorem~\ref{fano equality}, and
the organizers of the Conference on Moment Maps 2008 for a fun,
stimulating atmosphere. I would also like to thank my family and
friends for their support.

\end{document}